\documentclass[12pt]{article}
\usepackage{amsmath}
\usepackage{amssymb}
\usepackage{amsxtra}
\usepackage{amsthm}
\usepackage{mathrsfs}
\usepackage{mathtools}
\usepackage{graphicx}
\usepackage{epstopdf}
\usepackage{subfigure}
\usepackage{tabularx}
\usepackage{array}
\usepackage{float}
\usepackage{xcolor,graphicx,float}
\usepackage{caption}
\usepackage{hyperref}
\makeatletter
\def\EquationsBySection{\def\theequation
	{\thesection.\arabic{equation}}
	\@addtoreset{equation}{section}}

\def \d{{\mathrm{d}}}

\usepackage{indentfirst}
\newtheorem{remark}{Remark}[section]
\newtheorem{definition}[remark]{Definition}

\newtheorem{theorem}[remark]{Theorem}

\newtheorem{lemma}[remark]{Lemma}
\newtheorem{example}[remark]{Example}

\newcommand\old[1]{}
\makeatother
\usepackage{appendix}
\usepackage{geometry}
\allowdisplaybreaks[3]
\EquationsBySection
\begin{document}
	\date{}
	\title{The Onsager-Machlup action functional for Mckean-Vlasov SDEs\thanks{The work is partially supported  by the NSFC (Nos. 12171084, 12071071) and the Fundamental Research Funds for the Central Universities.}}
	\author{Shanqi Liu\footnote{School of Mathematical Science, Nanjing Normal University, Nanjing 210023, China} \thanks{E-mail: shanqiliumath@126.com} Hongjun Gao\footnote{School of mathematics, Southeast University, Nanjing, China, Nanjing 211189, China} \thanks{The correspondent author. E-mail: gaohj@hotmail.com}
   and Huijie Qiao\footnote{School of mathematics, Southeast University, Nanjing, China, Nanjing 211189, China}  \thanks{E-mail: hjqiaogean@seu.edu.cn}
		\\
	}
	\maketitle
	\noindent {\bf \small Abstract}{\small
		\quad This paper is devoted  to deriving the Onsager-Machlup  action functional for Mckean-Vlasov stochastic differential equations in a class of norms that dominate $L^2([0,1], \mathbb{R}^d)$, such as supremum norm $\|\cdot\|_{\infty}$, H$\mathrm{\ddot{o}}$lder norms $\|\cdot\|_{\alpha}$ with $\alpha<\frac{1}{4}$ and $L^p$-norms with $p>4$ are included.  Moreover, the corresponding Euler-Lagrange equation for Onsager-Machlup action functional is derived and a example is given.

		
		\vskip 0.2cm
		

		\section{Introduction}
			The dynamic of a stochastic dynamic system, especially the dynamics between the metastable states of the system is an important topic in the field of stochastic dynamic system. However, since the trajectories of the stochastic system are unpredictable, we need to find some deterministic and significant quantities that carry dynamical information or that can be used to quantify the dynamic behavior of stochastic differential equations (SDEs in short).  For example, mean exit time and escape probability, etc. One can refer to \cite{JD1, MI} for more details. For the Onsager-Machlup (OM) action functional, we are concerned about, it is also a deterministic quantity that characterizes the most probable transition path between the metastable states of the system by variational principle for the OM action function.
		
		Onsager-Machlup action functional was originally initiated by Onsager and Machlup \cite{OM1, OM2} as the probability density functional for diffusion processes with linear drift and constant diffusion coefficients, and then Tizsa and Manning \cite{LT} generalized the results of Onsager and Machlup to nonlinear equations. In 1957 the rigorous mathematical approach was proposed by Stratonovich \cite{RL} and further improved by Ikeda and Watanabe \cite{NI}, Takahashi and Watanabe \cite{YT}, Fujita and Kotani \cite{TF}. The key point was to express the transition probability of a diffusion process employing a functional integral over paths of the process, and the integrand was called the OM action function.
		
		Onsager-Machlup action functionals for classical stochastic differential equations driven by Brownian motions have been extensively investigated in past decades \cite{MC1,MC2,DB,TF,KH,KH1,SM,NI,RL,LS,YT}.  Ikeda and Watanabe \cite{NI} derived the OM action functional for reference path $\phi\in C^2([0,1],\mathbb{R}^{d})$ under the supremum norm $\|\cdot\|_{\infty}$. D\"{u}rr and Bach \cite{DB} obtained the same results based on the Girsanov transformation of the quasi-translation invariant measure and the potential function (path integral representation). Shepp and Zeitouni \cite{LS} proved that the result still holds for every norm equivalent to the supremum norm and $\phi-x$ in the Cameron-Martin space. Capitaine \cite{MC1}  extended this result to a large class of natural norms on the Wiener space including the H\"{o}lder norm, the Sobolev norm and the Besov norm. In \cite{KH}, Hara and Takahashi provided a computation of the Onsager-Machlup action functional of an elliptic diffusion process for the supremum norm and this result was extended in \cite{MC2} by Capitaine to norms that dominate supremum norm. In particular, the norms $\|\cdot\|$ could be any norm dominating $L^2$-norm on $\mathbb{R}^d$.

		The work on SDEs with coefficients that depend on the law of the solution was initiated by
		McKean \cite{MC}. And it has developed quickly in the past twenty
		years. Compared with classical SDEs, McKean-Vlasov SDEs (also called distribution dependent SDEs or mean-field SDEs) can describe the macro environment and micro-actions of the models. Up to now, the theory of  McKean-Vlasov SDEs have been extensively studied by a large number of researchers. For example, Wang \cite{WF1} proved the strong
		well-posedness and some functional inequalities for McKean-Vlasov SDEs. Recently, Ren and Wang \cite{PR} studied the path-independence of additive functionals for
		McKean-Vlasov SDEs, and then Qiao and Wu \cite{QH} extended it to McKean-Vlasov SDEs with jumps. Moreover, in \cite{PR} they also derived the It\^o formula for distribution dependent functions involving the $L$-derivative introduced by  P.L. Lions \cite{PC} under some conditions. This formula is important in this paper.
		
		In this paper, we are interested in deriving the OM action functional for Mckean-Vlasov SDEs. The difficulty in proving the main result (Theorem 3.1) is that we can not apply usual Taylor expansion technique to a stochastic integral as that 
		in \cite{NI}. In order to effectively overcome this difficulty, we use the It\^o formula for distribution dependent functions to deal with the stochastic integral. This technique perfectly avoids the dilemma after using the classical Taylor expansion. 
		And then we obtain the desired result by a lot of estimation.  Finally, we also give the corresponding Euler-Lagrange equation for the Onsager-Machlup action functional and apply it to a simple example.
		
	The remainder of this paper is organized as follows.  In Section $2$, we recall some basic notations, assumptions and lemmas. And then our main result is stated and proved in Section $3$. In Section 4, we derive the corresponding 
	Euler-Lagrange equation for Onsager-Machlup action functional and apply it to an example.
		
		\section{Framework}
	In this section, we recall some basic notations, assumptions, lemmas that will be used later.
	       \subsection{Notations}
	       Let $\mathscr{P}$ be the space of all probability measures $\mu$ in $\mathbb{R}^{d}$, and let
	$$\mathscr{P}_{2}(\mathbb{R}^{d})=\left\{\mu \in \mathscr{P}\left(\mathbb{R}^{d}\right): \mu\left(|\cdot|^{2}\right):=\int_{\mathbb{R}^{d}}|x|^{2} \mu(\mathrm{d} x)<\infty\right\}.$$
	It is well known that $\mathscr{P}_{2}$ is a Polish space under the Wasserstein distance
	$$
	\mathbb{W}_{2}(\mu, \nu):=\inf _{\pi \in \mathscr{C}(\mu, \nu)}\left(\int_{\mathbb{R}^{d} \times \mathbb{R}^{d}}|x-y|^{2} \pi(\mathrm{d} x, \mathrm{d} y)\right)^{\frac{1}{2}}, \mu, \nu \in \mathscr{P}_{2}\left(\mathbb{R}^{d}\right),
	$$
	where $\mathscr{C}(\mu, \nu)$ is the set of couplings for $\mu$ and $\nu$; that is, $\pi \in \mathscr{C}(\mu, \nu)$ is a probability measure on $\mathbb{R}^{d} \times \mathbb{R}^{d}$ such that $\pi\left(\cdot \times \mathbb{R}^{d}\right)=\mu$ and $\pi\left(\mathbb{R}^{d} \times \cdot\right)=\nu$.
	\begin{remark}
	(i)	For any $\mathbb{R}^d$-valued random variable $X$ and $Y$, it holds that
		\begin{align}
			\mathbb{W}_{2}(\mathscr{L}_{X}, \mathscr{L}_{Y})\leq [E|X-Y|^2]^{\frac{1}{2}}.
		\end{align}
	(ii) If $\phi_t$ is a                                                                                                                                                                                                                                                                                                                                                                                                                                                                                                                                                                                                                                                                                                                                                                                                                                                                                                                                                                                                                                                                                                                                                                                                                                                                                                                                                                                                                                                                                                                                                                                                                                                                                                                                                                                                                                                                                                                                                                                                                                                                                                                                                                                                                                                                                                                                                                                                                                                                                                                                                                                                                                                                                                                                                                                                                                                                                                                                                                                                                                                                                                                                                                                                                                                                                                                                                                                                                                                                                                                                                                                                                                                                                                                                                                                                                                                                                                                                                                                                                                                                                                                                                                                                                                                                                                                                                                                                                                                                                                                                                                                                                                                                                                                                                                                                                                                                                                                                                                                                                                                                                                                                                                                                                                                                                                                                                                                                                                                                                                                                                                                                                                                                                                                                                                                                                                                                                                                                                                                                                                                                                                                                                                                                                                                                                                                                                                                                                                                                                                                                                                                                                                                                                                                                                                                                                                                                                                                                                                                                                                                                                                                                                                                                                                                                                                                                                                                                                                                                                                                                                                                                                                                                                                                                                                                                                                                                                                                                                                                                                                                                                                                                                                                                                                                                                                                                                                                                                                                                                                                                                                                                                                                                                                                                                                                                                                                                                                                                                                                                                                                                                                                                                                                                                                                                                                                                                                                                                                                                                                                                                                                                                                                                                                                                                                                                                                                                                                                                                                                                                                                                                                                                                                                                                                                                                                                                                                                                                                                                                                                                                                                                                                                                                                                                                                                                                                                                                                                                                                                                                                                                                                                                                                                                                                                                                                                                                                                                                                                                                                                                                                                                                                                                                                                                                                                                                                                                                                                                                                                                                                                                                                                                                                                                                                                                                                                                                                                                                                                                                                                                                                                                                                                                                                                                                                                                                                                                                                                                                                             deterministic path, then the law of $\phi_t$ is a Delta measure at $\phi_t$, i.e. $\mathscr{L}_{\phi_t}=\delta_{\phi_t}$.
	\end{remark}
	\begin{definition} \cite{PR}
		 Let $T \in(0, \infty]$, and set $[0, T]=[0, \infty)$ when $T=\infty$.
	
		(i) A function $h: \mathscr{P}_{2}\left(\mathbb{R}^{d}\right) \rightarrow \mathbb{R}$ is called $L$-differentiable at $\mu \in \mathscr{P}_{2}\left(\mathbb{R}^{d}\right)$, if the functional
		$$
		L^{2}\left(\mathbb{R}^{d} \rightarrow \mathbb{R}^{d}, \mu\right) \ni \phi \mapsto h\left(\mu \circ(\operatorname{Id}+\phi)^{-1}\right)
		$$
		is Fréchet differentiable at $\phi=0 \in L^{2}\left(\mathbb{R}^{d} \rightarrow \mathbb{R}^{d}, \mu\right)$; that is, there exists (hence, unique) $\xi \in L^{2}\left(\mathbb{R}^{d} \rightarrow \mathbb{R}^{d}, \mu\right)$ such that
		$$
		\lim _{\mu\left(|\phi|^{2}\right) \rightarrow 0} \frac{h\left(\mu \circ(\mathrm{Id}+\phi)^{-1}\right)-h(\mu)-\mu(\langle\xi, \phi\rangle)}{\sqrt{\mu\left(|\phi|^{2}\right)}}=0 .
		$$
		In this case, we denote $\partial_{\mu} h(\mu)=\xi$ and call it the $L$-derivative of $h$ at $\mu$.
		
		(ii) A function $h: \mathscr{P}_{2}\left(\mathbb{R}^{d}\right) \rightarrow \mathbb{R}$ is called $L$-differentiable on $\mathscr{P}_{2}\left(\mathbb{R}^{d}\right)$ if the $L$-derivative $\partial_{\mu} h(\mu)$ exists for all $\mu \in \mathscr{P}_{2}\left(\mathbb{R}^{d}\right)$. If moreover $\left(\partial_{\mu} h(\mu)\right)(y)$ has a version differentiable in $y \in \mathbb{R}^{d}$ such that $\left(\partial_{\mu} h(\mu)\right)(y)$ and $\partial_{y}\left(\partial_{\mu} h(\mu)\right)(y)$ are jointly continuous in $(\mu, y) \in$ $\mathscr{P}_{2}\left(\mathbb{R}^{d}\right) \times \mathbb{R}^{d}$, we denote $h \in C^{(1,1)}\left(\mathscr{P}_{2}\left(\mathbb{R}^{d}\right)\right)$.
		
		(iii) A function $h:[0, T] \times \mathbb{R}^{d} \times \mathscr{P}_{2}\left(\mathbb{R}^{d}\right) \rightarrow \mathbb{R}$ is said to be in the class $C^{1,2,(1,1)}([0, T] \times$ $\mathbb{R}^{d} \times \mathscr{P}_{2}\left(\mathbb{R}^{d}\right)$ ), if the derivatives
		$$
		\partial_{t} h(t, x, \mu), \partial_{x} h(t, x, \mu), \partial_{x}^{2} h(t, x, \mu), \partial_{\mu} h(t, x, \mu)(y), \partial_{y} \partial_{\mu} h(t, x, \mu)(y)
		$$
		exist and are jointly continuous in the corresponding arguments $(t, x, \mu)$ or $(t, x, \mu, y)$. If $f \in C^{1,2,(1,1)}\left([0, T] \times \mathbb{R}^{d} \times \mathscr{P}_{2}\left(\mathbb{R}^{d}\right)\right)$ with all these derivatives bounded on $[0, T] \times$ $\mathbb{R}^{d} \times \mathscr{P}_{2}\left(\mathbb{R}^{d}\right)$, we denote $f \in C_{b}^{1,2,(1,1)}\left([0, T] \times \mathbb{R}^{d} \times \mathscr{P}_{2}\left(\mathbb{R}^{d}\right)\right)$.
	
		(iv) Finally, we write $h \in \mathscr{C}\left([0, \infty) \times \mathbb{R}^{d} \times \mathscr{P}_{2}\left(\mathbb{R}^{d}\right)\right)$, if $h \in C^{1,2,(1,1)}\left([0, T] \times \mathbb{R}^{d} \times \mathscr{P}_{2}\left(\mathbb{R}^{d}\right)\right)$ and the function
		$$
		(t, x, \mu) \mapsto \int_{\mathbb{R}^{d}}\left\{\left\|\partial_{y} \partial_{\mu} h\right\|+\left\|\partial_{\mu} h\right\|^{2}\right\}(t, x, \mu)(y) \mu(\mathrm{d} y)
		$$
		is locally bounded, i.e. it is bounded on compact subsets of $[0, T] \times \mathbb{R}^{d} \times \mathscr{P}_{2}\left(\mathbb{R}^{d}\right)$.
	\end{definition}

		\subsection{Mckean-Vlasov SDEs}
		In this paper, we are concerned with the following Mckean-Vlasov SDE in $\mathbb{R}^d$:
		\begin{align}
			\d X_t=f(t,X_t,\mathscr{L}_{X_{t}})\d t+g(t,X_t,\mathscr{L}_{X_{t}})\d B_t, X_0=x,
		\end{align}
	where $f: [0,1]\times \mathbb{R}^{d} \times \mathscr{P}_{2}\left(\mathbb{R}^{d}\right) \rightarrow \mathbb{R}^{d}, g:  [0,1]\times\mathbb{R}^{d} \times \mathscr{P}_{2}\left(\mathbb{R}^{d}\right) \rightarrow \mathbb{R}^{d \times d}, B_{t}$ is a $d$-dimensional Brownian motion and $\mathscr{L}_{X_{t}}$ is the law of $X_{t}$ under the given complete filtration probability
	space $(\Omega,\mathcal{F},(\mathcal{F}_t)_{t\ge0},\mathbb{P})$.
	
Next, we impose some assumptions throughout this paper. And the norm $\|\cdot\|$ stand for a class of norms that dominate $L^2([0,1],\mathbb{R}^d)$ and satisfy the following assumptions $\mathbf{(H2)}, \mathbf{(H3)}$.

$\bullet$ $\mathbf{(H1)}$
There exists an increasing function $K:[0,\infty)\mapsto (0,\infty)$ such that

(i) $$|f(t,x,\mu)-f(t,y,\nu)|+\|g(t,x,\mu)-g(t,y,\nu))\|_{HS}\leq K(t)\Big(|x-y|+\mathbb{W}_{2}(\mu,\nu)\Big),$$
for $t\in[0,1]$, $x,y\in\mathbb{R}^d,\nu,\mu\in \mathscr{P}_{2}\left(\mathbb{R}^{d}\right)$.

(ii)
$$\|g(t,0,\delta_0))\|_{HS}+|f(t,0,\delta_0))|\leq K(t),$$
where $t\in [0,1], 0\in \mathbb{R}^d$ and $\delta_0$ is the Delta measure at 0.

(iii) there exists a constant $C>0$ such that
$$|\partial_{x}f(t,x,\mu)-\partial_{x}f(t,y,\nu)|\leq C\Big(|x-y|+\mathbb{W}_{2}(\mu,\nu)\Big),$$
for $t\in[0,1]$, $x,y \in\mathbb{R}^d,\nu,\mu\in \mathscr{P}_{2}\left(\mathbb{R}^{d}\right)$.

$\bullet$ $\mathbf{(H2)}$ (i) The norm $\|\cdot\|$ is invariant under the action of the orthogonal group $\mathcal{O}(\mathbb{R}^d)$ on the coordinates of the Brownian motion, i.e.
$$
\left\|\left(B^1, \ldots, B^i, \ldots, B^d\right)\right\|=\left\|\left(B^1, \ldots,-B^i, \ldots, B^d\right)\right\|.
$$

(ii) For every $1\leq i\leq d$ and every $c\in \mathbb{R}$,
\begin{align}
	\limsup _{\varepsilon \rightarrow 0} E[\exp(c|B^i_1|^2)\vert \|B\|\leq \varepsilon]\leqslant 1.\nonumber
\end{align}

$\bullet$  $\mathbf{(H3)}$ There exists $0<q<p$ such that for any $\varepsilon$ small enough

(i) Under the condition $\|B\|\leq\varepsilon$, then
$$\int_{0}^{1}B^4_t\d t\leq C_1\varepsilon^p,$$
and 

(ii)
$$P(\|B\|\leq\varepsilon)\geq\exp(\frac{C_2}{\varepsilon^q}),$$
where $C_1, C_2$ are positive constants.
\begin{remark}
	 Here we explain assumptions $\mathbf{(H1)}$-$\mathbf{(H3)}$.
	 
	 $\mathbf{(a)}$ As is well known, by $\mathbf{(H1)}$ $(i), (ii)$, there exists a unique solution $(X_t)$ to $(2.1)$. And
	 the imposed condition $(iii)$ in $\mathbf{(H1)}$ is  to extract divergence part of OM  action functional. 
	 
	 $\mathbf{(b)}$  $\mathbf{(H2)},\mathbf{(H3)}$ are the usual assumptions, one can found similar assumptions in \cite{XB2,MC1,SM, LS1}.
	 More specifically, $\mathbf{(H2)}$ (i) is to satisfy the conditions of Lemma $2.7$. And $\mathbf{(H2)}$ (ii) is similar to  $(P2)$ in \cite{MC1}. And note that if $\|\cdot\|$ dominates the supremum norm, $\mathbf{(H2)}$ (ii) is immediate. 
	 
	 $\mathbf{(c)}$ In the classic case, in order to deal with remainder terms after using Taylar expansion, $\mathbf{(H3)}$ is the critical assumption of small ball probability. Although we choose to apply It$\mathrm{\hat{o}}$ formula's for distribution dependent function rather classical Taylor expansion to deal with stochastic integral, this assumption is still indispensable in our proof. Moreover, this assumption covers many class of norms, such as supremum norm $\|\cdot\|_{\infty}$, H$\mathrm{\ddot{o}}$lder norms $\|\cdot\|_{\alpha}$ with $\alpha<\frac{1}{4}$ and $L^p$-norms with $p>4$. One can refer Lifshits \cite{MA} for more detailed about estimation of small ball probability under different norms.
\end{remark}

\subsection{Onsager-Machlup action functional}
We now introduce the definition of Onsager-Machlup action functional for Mckean-Vlasov SDEs.
		\begin{definition}
			Let $\varepsilon>0$ be given. Consider a tube surrounding the reference path $\phi_t$, if for $\varepsilon$ sufficiently small we estimate the probability of the solution process $X_t$ of Mckean-Vlasov SDE $(2.1)$
			lying in this tube in the form:
			$$\mathbb{P}(\{\|X-\phi\|\le\varepsilon\})\propto C(\varepsilon)\exp\{\int_{0}^{1}\mathnormal{OM}(t, \phi_t, \dot{\phi}_t, \mathscr{L}_{\phi_t})\mathrm{d}t\},$$
			then integrand $\mathnormal{OM}(t,\phi_t,\dot{\phi}_t, \mathscr{L}_{\phi_t})$ is called $\text{\bf{Onsager-Machlup action function}}$ for Mckean-Vlasov SDEs. Where $\propto$ denotes the equivalence relation for $\varepsilon$ small enough and $\|\cdot\|$ is a suitable norm. We also call $\int_{0}^{1}\mathnormal{OM}(t, \phi_t, \dot{\phi}_t, \mathscr{L}_{\phi_t})\mathrm{d}t$ the $\text{\bf{Onsager-Machlup action functional}}$ for Mckean-Vlasov SDEs.
		\end{definition}
	In our paper, we always let the norm $\|\cdot\|$ stands for a class of norms that dominate $L^2([0,1],\mathbb{R}^d)$ and satisfy assumptions above. In addition, throughout the paper we assume reference path $\phi_t$ with initial value $\phi_0=x$ and $\phi_{t}-x$ belongs to Cameron-Martin space $\mathcal{H}$.  Cameron-Martin space $\mathcal{H}$ stands for  the class  of all absolutely continuous functions $h$
	such that $h(0) = 0$ and $\dot{h}\in L^2([0,1],\mathbb{R}^d)$; the inner product is given by
	the formula
	$$(h_1, h_2)_{\mathcal{H}}:=\int_{0}^{1}\dot{h}_1(t)\dot{h}_2(t)\d t.$$
		\subsection{Technical lemmas}
		\begin{lemma}[\cite{NI} \text{pp 536-537}]\label{lem1}
			For a fixed $n\geqslant1$, let $I_1, \ldots, I_n $ be $n$ random variables defined on $(\Omega,\mathcal{F},\mathbb{P})$ and $\{A_\varepsilon;\varepsilon >0\}$ a family of sets in $\mathcal{F}$. Suppose that for any $c\in \mathbb{R}$ and any $i=1,\dots, n$, if we have
			\begin{align}
				\limsup _{\varepsilon \rightarrow 0} E[\exp(cI_i)\vert A_\varepsilon]\leqslant 1.\nonumber
			\end{align}
			Then
			\begin{align}
				\lim_{\varepsilon\to 0} E \left [\exp \left(\sum_{i=1}^{n}cI_i \right ) \bigg| A_\varepsilon \right ]= 1.\nonumber
			\end{align}
		\end{lemma}
		\begin{lemma}\cite{LS}\label{lem1}
			Let $f$ be a deterministic function in $L^2[0,1]$. Define $I_i(f)=\int_{0}^{1}f(t)\d B^i_t$. If the norm $\|\cdot\|$ dominates the $L^1$-norm then
			\begin{align}
				\lim_{\varepsilon\to 0} E [\exp(|I_i(f)|)|\|				B\|<\varepsilon]= 1.\nonumber
			\end{align}
		\end{lemma}
		\begin{lemma}[\cite{LS1} \text{Theorem 1}]
			Assume that the norm $\|\cdot\|$ dominates the $L^{1}$-norm and satisfies
			$$
			\left\|\left(B^1, \ldots, B^i, \ldots, B^d\right)\right\|=\left\|\left(B^1, \ldots,-B^i, \ldots, B^d\right)\right\|.
			$$
			Let $\mathscr{F}_{i}$ be the $\sigma$-field generated by $\left\{B^1_t, \ldots, B^{i-1}_t,B^{i+1}_t, \ldots, B^{d}_t ; 0 \leq\right.$ $t \leq 1\} .$ Let $\Psi(\cdot)$ be an $\mathscr{F}_{i}$-adapted function such that
			$$
			\forall c \in \mathbb{R}^{+}, \lim _{\varepsilon \rightarrow 0} E\left(\exp \left\{c \int_{0}^{1} \Psi^{2}(t) \d t\right\} \mid\|B\|<\epsilon\right)=1.
			$$
			Then
			$$
			\lim _{\varepsilon \rightarrow 0} E\left(\exp \left\{\left|\int_{0}^{1} \Psi(t) \d B_t^i\right|\right\} \mid\|B\|<\epsilon\right)=1.
			$$
		\end{lemma}
	\begin{lemma}[\cite{PR} \text{Lemma 3.1}]\text{(It$\mathrm{\hat{o}}$ formula's  for distribution dependent function)}
		Assume $X_t$ satisfies $(2.1)$, then for any $h \in \mathscr{C}([0, \infty) \times$ $\left.\mathbb{R}^{d} \times \mathscr{P}_{2}\left(\mathbb{R}^{d}\right)\right), h\left(t, X_{t}, \mathscr{L}_{X_{t}}\right)$ is a semi-martingale with
		$$
		\mathrm{d} h\left(t, X_{t}, \mathscr{L}_{X_{t}}\right)=\left(\partial_{t}+\mathbf{L}_{f, g}\right) h\left(t, X_{t}, \mathscr{L}_{X_{t}}\right) \mathrm{d} t+\left\langle\left(g^{*} \partial_{x} h\right)\left(t, X_{t}, \mathscr{L}_{X_{t}}\right), \mathrm{d} B_{t}\right\rangle,
		$$
		where
		$$
		\begin{aligned}
			\mathbf{L}_{f, g} h(t, x, \mu)=& \frac{1}{2} \operatorname{tr}\left(g g^{*} \partial_{x}^{2} h\right)(t, x, \mu)+\left\langle f, \partial_{x} h\right\rangle(t, x, \mu) \\
			&+\int_{\mathbb{R}^d}\Big[\frac{1}{2} \operatorname{tr}\left\{\left(g g^{*}\right)(t, y, \mu) \partial_{y} \partial_{\mu} h(t, x, \mu)(y)\right\}+\langle f(t,y,\mu), \partial_{\mu}h(t,x,\mu)(y)\rangle\Big]\mu(\d y).
		\end{aligned}
	$$
	\end{lemma}
		\section{Main results and Proofs}
		Limited to the fact that the Girsanov transform cannot fully extract the information of the OM action functional when $g$ is not a constant diffusion matrix. For this reason and without loss of generality, we consider the following Mckean-Vlasov SDE with constant diffusion matrix $g=I_{d\times d}$:
			\begin{align}
			dX_t=f(t,X_t,\mathscr{L}_{X_{t}})\d t+\d B_t, X_0=x.
		\end{align}
	
	We now state our main result associated with OM action functional for $(3.1)$.
	\begin{theorem}
		Assume that $f \in C_{b}^{1,2,(1,1)}\left([0, 1] \times \mathbb{R}^{d} \times \mathscr{P}_{2}\left(\mathbb{R}^{d}\right)\right)$ and $\mathbf{(H1)}-\mathbf{(H3)}$ hold. Reference path $\phi_t-x$ belongs to Cameron-Martin $\mathcal{H}$. Then the Onsager-Machlup action functional of $X_t$ for any norm dominating $L^2([0,1],\mathbb{R}^d)$  exists and is given by
		$$L(t,\phi,\dot{\phi},\delta_{\phi})=-\frac{1}{2}\int_{0}^{1}|\dot{\phi}_t-f(t,\phi_t,\mathscr{L}_{\phi_t})|^2\ dt-\frac{1}{2}\int_{0}^{1}\operatorname{div}_{x}f(t,\phi_t,\mathscr{L}_{\phi_t})\d t,$$
		where $\operatorname{div}_{x}$  denote the divergence on the $\phi_t\in\mathbb{R}^d$.
	\end{theorem}
\begin{proof}
	Let $Y_t$ be the solution for following SDE
	$$Y_t=\phi_t+B_t, Y_0=x\in\mathbb{R}^d.$$
	Due to $f \in C_{b}^{1,2,(1,1)}\left([0, 1] \times \mathbb{R}^{d} \times \mathscr{P}_{2}\left(\mathbb{R}^{d}\right)\right)$ and $\phi_t\in \mathcal{H}$, the Novikov condition is clearly satisfied. Girsanov theorem imply that $\hat{B}_t=B_t-\int_{0}^{t}[f(u,Y_u,\mathscr{L}_{Y_{u}})-\dot{\phi}_u]\d u$ is a $d$-dimensional Brownian motion under new probability  defined by $\frac{d\mathbb{Q}}{d\mathbb{P}}:=R$ with
	$$R:=\exp\Big( \int_{0}^{1}\langle f(t,Y_t,\mathscr{L}_{Y_{t}})-\dot{\phi}_t, \d B_t\rangle-\frac{1}{2}\int_{0}^{1}|f(t,Y_t,\mathscr{L}_{Y_{t}})-\dot{\phi}_t|^2\d t\Big).$$
	 So we have
	\begin{align}
		\frac{\mathbb{P}(\|X-\phi\|\leq\varepsilon)}{\mathbb{P}(\|B\|\leq\varepsilon)}&=\frac{\mathbb{Q}(\|Y-\phi\|\leq\varepsilon)}{\mathbb{P}(\|B\|\leq\varepsilon)}=\frac{E\Big(R \mathbb{I}_{\|B
				\|\leq\varepsilon}\Big)}{\mathbb{P}(\|B\|\leq\varepsilon)}=E\Big(R|\|B\|\leq\varepsilon\Big)\nonumber\\&= E\Big(\exp\Big(\int_{0}^{1}\langle f(t,Y_t,\mathscr{L}_{Y_{t}})-\dot{\phi}_t, \d B_t\rangle-\frac{1}{2}\int_{0}^{1}|f(t,Y_t,\mathscr{L}_{Y_{t}})-\dot{\phi}_t|^2\d t\Big)|\|B\|\leq\varepsilon\Big)\nonumber\\&:=\exp\Big(-\frac{1}{2}\int_{0}^{1}|\dot{\phi}_t-f(t,\phi_t,\mathscr{L}_{\phi_{t}})|^2\d t\Big)E\Big(\exp(\sum_{i=1}^{4}T_i)|\|B\|\leq\varepsilon\Big),
	\end{align}
where
\begin{align}
	T_1&:=\int_{0}^{1}\langle f(t,\phi_t+B_t,\mathscr{L}_{\phi_t+B_t})-f(t,\phi_t,\mathscr{L}_{\phi_t}),\dot\phi_t\rangle \mathrm{d}t,\nonumber\\
	T_2&:=\frac{1}{2}\int_{0}^{1}\Big|f(t,\phi_t,\mathscr{L}_{\phi_t})\Big|^2\mathrm{d}t-\frac{1}{2}\int_{0}^{1}\Big|f(t,\phi_t+B_t,\mathscr{L}_{\phi_t+B_t})\Big|^2\mathrm{d}t,\nonumber\\
	T_3&:=\int_{0}^{1}\langle-\dot{\phi}_t,\mathrm{d}B_t\rangle,\nonumber\\T_4&=\int_{0}^{1}\langle f(t,\phi_t+B_t,\mathscr{L}_{\phi_t+B_t}),\mathrm{d}B_t\rangle.\nonumber
\end{align}

Firstly, we deal with the terms $T_1$ and $T_2$. By $\mathbf{(H1)}(i)$, $(2.2)$, H$\mathrm{\ddot{o}}$lder inequality, Fubini theorem and the assumptions on $f$ and $\phi$, we have the following estimates on the set $\{\|B\|\leq \varepsilon\}$.
\begin{align}
	|T_1|&\leq \int_{0}^{1}|\dot{\phi}_t|K(t)\Big(|B_t|+\mathbb{W}_2(\mathscr{L}_{\phi_t+B_t},\mathscr{L}_{\phi_t})\Big)\d t\leq C_3\Big[\Big(\int_{0}^{1}|B_t|^2\d t\Big)^{\frac{1}{2}}+\Big(\int_{0}^{1}E(B_t^2)\d t\Big)^{\frac{1}{2}}\Big]\leq 2C_3 \varepsilon ,\nonumber\\|T_2|&\leq C_4\Big[\Big(\int_{0}^{1}|B_t|^2\d t\Big)^{\frac{1}{2}}+\Big(\int_{0}^{1}E(B_t^2)\d t\Big)^{\frac{1}{2}}\Big]\leq 2C_4\varepsilon.\nonumber
\end{align}
So we easily get
\begin{align}
	\limsup _{\varepsilon \rightarrow 0}\ E(\exp(cT_{1})|\|B\|<\varepsilon)\leq1,
\end{align}
\begin{align}
	\limsup _{\varepsilon \rightarrow 0}\ E(\exp(cT_{2})|\|B\|<\varepsilon)\leq1,
\end{align}
for  every real number $c$.

We proceed to show that
\begin{align}
	\limsup _{\varepsilon \rightarrow 0}\ E(\exp(cT_{3})|\|B\|<\varepsilon)\leq1,
\end{align}
for  every real number $c$. It is clear by  $\phi_t\in\mathcal{H}$ and Lemma $2.6$.

Therefore, by Lemma $2.5$, we obtain $T_1, T_2, T_3$ have no contributions on the expression of OM action functional, i.e.
\begin{align}
		\limsup _{\varepsilon \rightarrow 0}\ E(\exp(c\sum_{i=1}^{3}T_{i})|\|B\|<\varepsilon)\leq1,
\end{align}
for  every real number $c$.

We now deal with the term $T_4$. Applying Lemma $2.8$ to $f_i(t,\phi_t+B_t,\mathscr{L}_{\phi_t+B_t})B_t^i$, we get
\begin{align}
\int_{0}^{1}f_i(t,\phi_t+B_t,\mathscr{L}_{\phi_t+B_t})\d B_t^i&= f_i(1,\phi_1+B_1,\mathscr{L}_{\phi_1+B_1})B_1^i-\int_{0}^{1}B_t^i(\partial_{t}+\mathbf{L}_{1}^{i})f_i(t,\phi_t+B_t,\mathscr{L}_{\phi_t+B_t})\d t\nonumber\\-&\sum_{j=1}^{d}\int_{0}^{1}B_t^i\partial_{x_j}f_i(t,\phi_t+B_t,\mathscr{L}_{\phi_t+B_t})\d B_t^j-\int_{0}^{1}\partial_{x_i}f_i(t,\phi_t+B_t,\mathscr{L}_{\phi_t+B_t})\d t,
\end{align}
where
$$
\begin{aligned}
	\mathbf{L}^i_{1} f_i(t, x, \mu)=& \frac{1}{2} \operatorname{tr}\left( \partial_{x}^{2} f_i\right)(t, x, \mu)+\left\langle \dot{\phi}, \partial_{x} f_i\right\rangle(t, x, \mu) \\
	&+\int_{\mathbb{R}^d}\Big[\frac{1}{2} \operatorname{tr}\left\{ \partial_{y} \partial_{\mu} f_i(t, x, \mu)(y)\right\}+\langle \dot{\phi}_t, \partial_{\mu}f_i(t,x,\mu)(y)\rangle\Big]\mu(\d y).
\end{aligned}
$$
So we can rewrite $T_4$ as
\begin{align}
	T_4&= \sum_{i=1}^df_i(1,\phi_1+B_1^i,\mathscr{L}_{\phi_1+B_1})B_1^i-\sum_{i=1}^d\int_{0}^{1}B_t^i(\partial_{t}+\mathbf{L}_{1}^{i})f_i(t,\phi_t+B_t,\mathscr{L}_{\phi_t+B_t})\d t\nonumber\\&-\sum_{i=1}^d\sum_{j=1}^{d}\int_{0}^{1}B_t^i\partial_{x_j}f_i(t,\phi_t+B_t,\mathscr{L}_{\phi_t+B_t})\d B_t^j-\sum_{i=1}^{d}\int_{0}^{1}\partial_{x_i}f_i(t,\phi_t+B_t,\mathscr{L}_{\phi_t+B_t})\d t\nonumber\\&:=T_5+T_6+T_7+T_8,
\end{align}
where
$$
\begin{aligned}
	T_5=&\sum_{i=1}^df_i(1,\phi_1+B_1^i,\mathscr{L}_{\phi_1+B_1})B_1^i-\sum_{i=1}^d\int_{0}^{1}B_t^i(\partial_{t}+\mathbf{L}_{1}^{i})f_i(t,\phi_t+B_t,\mathscr{L}_{\phi_t+B_t})\d t,\nonumber\\T_6=&-\sum_{i=1}^d\sum_{j=1}^{d}\int_{0}^{1}B_t^i\partial_{x_j}f_i(t,\phi_t,\mathscr{L}_{\phi_t})\d B_t^j-\sum_{i=1}^{d}\int_{0}^{1}\partial_{x_i}f_i(t,\phi_t,\mathscr{L}_{\phi_t})\d t,\nonumber\\T_7=&\sum_{i=1}^{d}\int_{0}^{1}\Big(\partial_{x_i}f_i(t,\phi_t+B_t,\mathscr{L}_{\phi_t+B_t})-\partial_{x_i}f_i(t,\phi_t,\mathscr{L}_{\phi_t})\Big)\d t,\\ T_8=&\sum_{i=1}^d\sum_{j=1}^{d}\int_{0}^{1}B_t^i\Big(\partial_{x_j}f_i(t,\phi_t,\mathscr{L}_{\phi_t})-\partial_{x_j}f_i(t,\phi_t+B_t,\mathscr{L}_{\phi_t+B_t})\Big)\d B_t^j.
\end{aligned}
$$
By Lemma $2.5$, we can deal with each term respectively. Since $f \in C_{b}^{1,2,(1,1)}\left([0, 1] \times \mathbb{R}^{d} \times \mathscr{P}_{2}\left(\mathbb{R}^{d}\right)\right)$, then under the condition $\|B\|\leq\varepsilon$ we have
\begin{align}
	\limsup _{\varepsilon \rightarrow 0}\ E(\exp(-c\sum_{i=1}^d\int_{0}^{1}B_t^i(\partial_{t}+\mathbf{L}_{1}^{i})f_i(t,\phi_t,\mathscr{L}_{\phi_t+B_t})\d t)|\|B\|<\varepsilon)\leq1,
\end{align}
for  every real number $c$.\\
By the boundness of $f_i$ and Lemma $2.6$, we have ($B^1_i=\int_{0}^{1}dB_t^i$)
\begin{align}
	\limsup _{\varepsilon \rightarrow 0}\ E(\exp(cf_i(1,\phi_1+B_1^i,\mathscr{L}_{\phi_1+B_1})B_1^i)|\|B\|<\varepsilon)\leq1,
\end{align}
for  $1\leq i\leq d$ and every real number $c$.

 So by Lemma $2.5$, we obtain
 \begin{align}
 	\limsup _{\varepsilon \rightarrow 0}\ E(\exp(cT_5)|\|B\|<\varepsilon)\leq1,
 \end{align}
 for  every real number $c$.

For the term $T_6$, we need to divide it into three parts:
$$-\sum_{i\neq j}^d\int_{0}^{1}\partial_{x_j}f_i(t,\phi_t,\mathscr{L}_{\phi_t})B_t^j\d B_t^i-\sum_{i=1}^{d}\int_{0}^{1}\partial_{x_i}f_i(t,\phi_t,\mathscr{L}_{\phi_t})B_t^i\d B_t^i-\sum_{i=1}^{d}\int_{0}^{1}\partial_{x_i}f_i(t,\phi_t,\mathscr{L}_{\phi_t})\d t.$$
\\
By $\mathbf{(H2)}$ $(i)$, Lemma $2.5$ and Lemma $2.7$, we can easily get
\begin{align}
	\limsup _{\varepsilon \rightarrow 0}\ E(\exp(-c\sum_{i\neq j}^d\int_{0}^{1}\partial_{x_j}f_i(t,\phi_t,\mathscr{L}_{\phi_t})B_t^j\d B_t^i)|\|B\|<\varepsilon)\leq1,
\end{align}
for  every real number $c$.

 Similarly, applying Lemma $2.8$ to $\partial_{x_i}f_i(t,\phi_t,\mathscr{L}_{\phi_t})(B_t^i)^2$ we obtain
 \begin{align}
 	\int_{0}^{1}\partial_{x_i}f_i(t,\phi_t,\mathscr{L}_{\phi_t})B_t^i\d B_t^i&= \frac{1}{2}\partial_{x_i}f_i(1,\phi_1,\mathscr{L}_{\phi_1})(B_1^i)^2-\frac{1}{2}\int_{0}^{1}(B_t^i)^2(\partial_{t}+\mathbf{L}_{2}^{i})f_i(t,\phi_t,\mathscr{L}_{\phi_t})\d t\nonumber\\&-\frac{1}{2}\int_{0}^{1}\partial_{x_i}f_i(t,\phi_t,\mathscr{L}_{\phi_t})\d t,
 \end{align}
 where
\begin{align}
		\mathbf{L}_2^{i}\partial_{x_i}f_i(t,x,\mu)=\langle \dot{\phi}, \partial_{x} \partial_{x_i}f_i\rangle(t, x, \mu) +\int_{\mathbb{R}^d}\langle\dot{\phi}_t, \partial_{\mu}[\partial_{x_i}f_i(t,x,\mu)](y)\rangle\mu(\d y).\nonumber
\end{align}
So we have
\begin{align}
	-\sum_{i=1}^d\int_{0}^{1}\partial_{x_i}f_i(t,\phi_t,\mathscr{L}_{\phi_t})B_t^i\d B_t^i&= -\sum_{i=1}^d\frac{1}{2}\partial_{x_i}f_i(1,\phi_1,\mathscr{L}_{\phi_1})(B_1^i)^2+\sum_{i=1}^d\frac{1}{2}\int_{0}^{1}(B_t^i)^2(\partial_{t}+\mathbf{L}_{2}^{i})f_i(t,\phi_t,\mathscr{L}_{\phi_t})\d t\nonumber\\&+\sum_{i=1}^d\frac{1}{2}\int_{0}^{1}\partial_{x_i}f_i(t,\phi_t,\mathscr{L}_{\phi_t})\d t.
\end{align}
Since the derivatives of $f_i$ are bounded, so by
 $\mathbf{(H2)}$(ii$)$, we have
 \begin{align}
 	\limsup _{\varepsilon \rightarrow 0}\ E(\exp(-\frac{c}{2}\partial_{x_i}f_i(1,\phi_1,\mathscr{L}_{\phi_1})(B_1^i)^2)|\|B\|<\varepsilon)\leq1,\nonumber
 \end{align}
 for $1\leq i\leq d$ and every real number $c$.

 And $f \in C_{b}^{1,2,(1,1)}\left([0, 1] \times \mathbb{R}^{d} \times \mathscr{P}_{2}\left(\mathbb{R}^{d}\right)\right)$ and $\phi\in\mathcal{H}$, we have
 \begin{align}
 	\limsup _{\varepsilon \rightarrow 0}\ E(\exp(\frac{c}{2}\int_{0}^{1}(B_t^i)^2(\partial_{t}+\mathbf{L}_{2}^{i})f_i(t,\phi_t,\mathscr{L}_{\phi_t})\d t)|\|B\|<\varepsilon)\leq1,\nonumber
 \end{align}
 for $1\leq i\leq d$ and every real number $c$.

 Therefore by Lemma $2.5$, we obtain
\begin{align}
	\limsup _{\varepsilon \rightarrow 0}\ E(\exp(c(T_6+\frac{1}{2}\int_{0}^{1}\operatorname{div}_{x}{f}(t,\phi_t,\mathscr{L}_{\phi_t})\d t))\|B\|<\varepsilon)\leq1,\nonumber
\end{align}
for  every real number $c$. And $-\frac{1}{2}\int_{0}^{1}\operatorname{div}_{x}{f}(t,\phi_t,\mathscr{L}_{\phi_t})\d t$ is the divergence part of OM action functional.

Moreover, on the set $\{\|B\|\leq \varepsilon\}$ using $\mathbf{(H1)}$$(iii)$, we have for $1\leq i\leq d$
\begin{align}
	\int_{0}^{1}|\partial_{x_i}f_i(t,\phi_t+B_t,\mathscr{L}_{\phi_t+B_t})-\partial_{x_i}f_i(t,\phi_t,\mathscr{L}_{\phi_t})|\d t\leq C_6\Big[\Big(\int_{0}^{1}|B_t|^2\d t\Big)^{\frac{1}{2}}+\Big(\int_{0}^{1}E(B_t^2)\d t\Big)^{\frac{1}{2}}\Big]\leq 2C_6\varepsilon,\nonumber
\end{align}
so by Lemma $2.5$, we get
\begin{align}
	\limsup _{\varepsilon \rightarrow 0}\ E(\exp(cT_7)\|B\|<\varepsilon)\leq1,\nonumber
\end{align}
for  every real number $c$.

Finally, we deal with term $T_{8}$ by estimation of small ball probability. We first define $T_8^{ij}:=\int_{0}^{1}B_t^i\Big(\partial_{x_j}f_i(t,\phi_t,\mathscr{L}_{\phi_t})-\partial_{x_j}f_i(t,\phi_t+B_t,\mathscr{L}_{\phi_t+B_t})\Big)\d B_t^j$ and rewrite $\mathbb{P}(|T^{ij}_{8}|>\delta|\|B\|\leq \varepsilon)$ as
\begin{align}
	\mathbb{P}(|T^{ij}_{8}|>\delta|\|B_t\|\leq \varepsilon)=\frac{\mathbb{P}(|T_{8}|>\delta, \|B\|\leq \varepsilon)}{\mathbb{P}(\|B\|\leq\varepsilon)}.\nonumber
\end{align}
Since $T_{8}^{ij}$ is a martingale, and its quadratic variations can be estimated by $\mathbf{(H1)}$$(iii)$, $\mathbf{H3}$ $(i)$ on the set $\{\|B\|\leq \varepsilon\}$ as follows,
\begin{align}
	\langle T^{ij}_{8} \rangle_t=&\int_{0}^{1}\Big|B_t^i\Big(\partial_{x_j}f_i(t,\phi_t,\mathscr{L}_{\phi_t})-\partial_{x_j}f_i(t,\phi_t+B_t,\mathscr{L}_{\phi_t+B_t})\Big)\Big|^2\d t\nonumber\\&\leq C_7\Big[\int_{0}^{1}|B_t|^4\d t+\int_{0}^{1}|B_t|^2E[B_t^2 ]\d t\Big]\nonumber\\&\leq C_7\Big[\frac{3}{2}\int_{0}^{1}|B_t|^4\d t+\frac{1}{2}\int_{0}^{1}E[B_t^4]\d t\Big] \leq 2C_{7}\varepsilon^{p}.
\end{align}
for $1\leq i,j \leq d$. So by the standard exponential inequality for martingales (similar pp196 in \cite{MC1} or pp1247 in \cite{XB2}), we have
\begin{align}
	\mathbb{P}(|T^{ij}_{8}|>\delta, \|B\|\leq \varepsilon)\leq \exp\Big(-\frac{\delta^2}{2C_{8}\varepsilon^{p}}\Big).
\end{align}
Combining assumption  $\mathbf{H 3}$ $(ii)$, by Lemma $2.5$ we get
\begin{align}
	\limsup _{\varepsilon \rightarrow 0}\ E(\exp(cT_{8})|\|B\|<\varepsilon)\leq1,
\end{align}
for every real number $c$.

As a consequence, by Lemma $2.5$ we have
\begin{align}
	\limsup _{\varepsilon \rightarrow 0}\ E(\exp(c(T_{4}+\frac{1}{2}\int_{0}^{1}\operatorname{div}_{x}{f}(t,\phi_t,\mathscr{L}_{\phi_t})\d t))|\|B\|<\varepsilon)\leq1,
\end{align}
for every real number $c$.

Recall previous works and then Lemma $2.5$ allows us to conclude that
$$
\lim _{\varepsilon \rightarrow 0} \frac{\mathbb{P}(\|X-\phi\|<\varepsilon)}{\mathbb{P}(\|B\|<\varepsilon)}=\exp \left(-\frac{1}{2} \int_{0}^{1}|\dot{\phi}_t-f(t,\phi_t,\mathscr{L}_{\phi_t})|^{2} \d t-\frac{1}{2} \int_{0}^{1}\operatorname{div}_xf(t,\phi_t,\mathscr{L}_{\phi_t})\right)\d t.
$$
The proof of Theorem $3.1$ is complete.
\end{proof}
\begin{remark}
	From the expression of OM action functional for Mckean-Vlasov SDEs, one can easily note that when the coefficient $f$ does not depend on the distribution of the solution, OM action functional coincides with the classical OM action functional.
\end{remark}
		\section{Euler-Lagrange Equations for OM action functional}
		In this section, we would derive Euler-Lagrange equations for the OM action function (functional). The basic idea is inspired by the proof of It$\mathrm{\hat{o}}$ formula for distribution dependent functions such as \cite{PR}.
		
		Let
		$$\overline{OM}(t,\phi_t,\dot{\phi}_t):=OM(t,\phi_t,\dot{\phi}_t,\delta_{\phi_t}).$$
		Then by the classical Euler-Lagrange equation we have
		$$\frac{\mathrm{d}}{\mathrm{d} t} \frac{\partial \overline{OM}(t,\phi_t,\dot{\phi}_t)}{\partial \dot{\phi}_t}=\frac{\partial \overline{OM}(t,\phi_t,\dot{\phi}_t)}{\partial \phi_t}.$$
	So we have (d=1)
			\begin{align}
				&\frac{\partial \overline{OM}(t,\phi_t,\dot{\phi}_t)}{\partial \dot{\phi}_t}=-(\dot{\phi}_t-\bar{f}(t,\phi_t))=(f(t,\phi_t,\delta_{\phi_t})-\dot{\phi}_t),\\&\frac{\partial \overline{OM}(t,\phi_t,\dot{\phi}_t)}{\partial \phi_t}=(\dot{\phi}_t-\bar{f}(t,\phi_t)\partial_{x}\bar{f}(t,\phi_t)-\frac{1}{2}\partial_{x}^2 \bar{f}(t,\phi_t)=(\dot{\phi}_t-f(t,\phi_t,\delta_{\phi_t})\partial_{x}f(t,\phi_t,\delta_{\phi_t})-\frac{1}{2}\partial_{x}^2f(t,\phi_t,\delta_{\phi_t}),
			\end{align}
			and
			\begin{align}
			\frac{\mathrm{d}}{\mathrm{d} t} \frac{\partial \overline{OM}(t,\phi_t,\dot{\phi}_t)}{\partial \dot{\phi}_t}&=\frac{\mathrm{d}}{\mathrm{d} t} \Big((\bar{f}(t,\phi_t)-\dot{\phi}_t)\Big)=\ddot{\phi}_t-	\partial_{x}\bar{f}(t,\phi_t)\dot{\phi}_t-\partial_{t}\bar{f}(t,\phi_t)\nonumber\\&=	\partial_{x}\bar{f}(t,\phi_t)\dot{\phi}_t-\partial_{t}f(t,\phi_t,\delta_{\phi_t})(\partial_{t}^{\delta_{\phi(t)}}f)(t,\phi_t,\dot{\phi}_t)-\ddot{\phi}_t\nonumber\\&=-\ddot{\phi}_t+	\partial_{x}f(t,\phi_t,\dot{\phi}_t)\dot{\phi}_t+\partial_{t}f(t,\phi_t,\delta_{\phi_t})+(\partial_{t}^{\delta_{\phi(t)}}f)(t,\phi_t,\dot{\phi}_t),
			\end{align}
		where ($\delta_{\phi_t}:=\mu^{\phi}_{t}$)
		\begin{align}
			(\partial_{t}^{\delta_{\phi(t)}}f)(t,\phi_t,\dot{\phi}_t):=\int_{\mathbb{R}}\dot{\phi}_t(\partial_{\mu_t^\phi}f)(t,\phi_t,\mu_t^{\phi})(y)\mu_t^\phi(\d y).
		\end{align}
	Combining $(4.1)-(4.3)$, we obtain Euler-Lagrange equation for OM action functional.
	\begin{align}
		\ddot{\phi}_t=\partial_{t}f(t,\phi_t,\delta_{\phi_t})+\frac{1}{2}\partial_{x}^2f(t,\phi_t,\delta_{\phi_t})+f(t,\phi_t,\delta_{\phi_t})\partial_{x}f(t,\phi_t,\delta_{\phi_t})+(\partial_{t}^{\delta_{\phi(t)}}f)(t,\phi_t,\dot{\phi}_t).
	\end{align}
	We can get a similar result for the multidimensional case.
		\begin{remark}
			In fact, when we are devoted to obtaining the most probable path for $(3.1)$ by Euler-Lagrange equation to OM action function, the uniformly bounded on $f$ is so strong. Since so many physical models are not satisfied such as $f(x,\mu)=\int_{\mathbb{R}}xy\mu(\d y)$. So an alternative approach is that we assume reference path $\phi \in {C}_b^2([0,1],\mathbb{R}^{d})$ and then let  $f \in \mathscr{C}\left([0, \infty) \times \mathbb{R}^{d} \times \mathscr{P}_{2}\left(\mathbb{R}^{d}\right)\right)$.
		\end{remark}	
Now we apply our results to the following example.
\begin{example}
	Consider the following scalar Mckean-Vlasov SDE:
\begin{align}
	\d X_t=E(X_t)[X_t-X_t^3]\d t+\d B_t, X_0=1,
\end{align}
where $f(x,\mu)=\int_{\mathbb{R}}(x-x^3)y\mu(\d y)$and $g(x,\mu)=1$. Let $f=0$ we can get $1,-1,0$ are metastable states of system $(4.6)$. And by Theorem $3.1$ and Remark $4.1$ we can obtain the Onsager-Machlup action functional for $(4.6):$
$$L(\phi_t,\dot{\phi}_t,\delta_{\phi_t})=-\frac{1}{2}\int_{0}^{1}|\dot{\phi}_t-\int_{\mathbb{R}}(\phi_t-\phi_t^3)y\mu_t^\phi(dy)|^2\d t-\frac{1}{2}\int_{\mathbb{R}}(1-3\phi_t^2)y\mu_t^\phi(\d y).$$

Then we can find the most probable path $\phi_t^*$ for $X_t$ by minimizing the corresponding OM action functional $L(\phi_t,\dot{\phi}_t, \delta_{\phi_t})$ with the help of variational principle.

By applying Euler-Lagrange equation $(4.5)$ for $L(\phi_t,\dot{\phi}_t,\delta_{\phi_t})$, we obtain
\begin{align}
\ddot{\phi}^*_t=\int_{\mathbb{R}}(-3\phi^*_t)y\mu_t^{\phi^*}(dy)+\int_{\mathbb{R}}(\phi^*_t-(\phi^*_t)^3)y\mu_t^{\phi^*}(\d y)\int_{\mathbb{R}}(1-3(\phi_t^*)^2)y\mu_t^{\phi^*}(\d y)+\phi_t^*-(\phi_t^*)^3,\nonumber
\end{align}
with boundary conditions $\phi^*_0:=1$ and $\phi^*_1=-1$. So if we can solve this ODEs, we immediately obtain the most probable path of system $(4.6)$ between two different metastable states. This is an important and interesting application of OM action functional.
\end{example}

			\vspace{0.5cm}
			%\noindent {\bf Acknowledgements.}

		\end{document}